\documentclass{amsart}

\usepackage{amsmath,amsthm,amssymb,latexsym,ifthen,graphicx}
\usepackage[colorlinks,citecolor=black,urlcolor=black, linkcolor=black, backref]{hyperref}
\usepackage{tikz}

\newcommand{\scmid}{\mathrel{}\middle|\mathrel{}}

\usepackage{xcolor}	
	\usepackage{soul}

\usepackage[utf8]{inputenc}
\usepackage{microtype}
\usepackage{mathrsfs}
\usepackage{undertilde}
\newtheorem{theorem}{Theorem}[section]

\newtheorem{corollary}[theorem]{Corollary}

\newtheorem{definition}[theorem]{Definition}

\newtheorem{First Proof}[theorem]{First Proof}

\newtheorem{Second Proof}[theorem]{Second Proof}

\newtheorem{lemma}[theorem]{Lemma}

\newtheorem{proposition}[theorem]{Proposition}
\newtheorem{question}[theorem]{Question}
\newtheorem{remark}[theorem]{Remark}

\theoremstyle{definition}

\newcommand{\equivT}{\equiv_{{\rm T}}}
\newcommand{\leqT}{\leq_{{\rm T}}}
\newcommand{\geqT}{\geq_{{\rm T}}}
\newcommand{\leT}{<_{{\rm T}}}
\newcommand{\geT}{>_{{\rm T}}}

\newcommand{\leqtt}{\leq_{{\rm tt}}}

\newcommand{\uh}{\!\upharpoonright\!}

\newcommand{\RCA}{\mathrm{RCA_0}}
\newcommand{\MPT}{\mathrm{MPT}}
\newcommand{\wMPT}{\mathrm{wMPT}}
\newcommand{\WKL}{\mathrm{WKL_0}}
\newcommand{\ACA}{\mathrm{ACA_0}}
\newcommand{\ATR}{\mathrm{ATR_0}}

\newcommand{\HK}{\mathcal{HK}}
\newcommand{\DET}{\mathrm{DET}}

\newcommand{\PST}{\mathrm{PST}}
\newcommand{\RPST}{\mathrm{RPST}}
\newcommand{\CPST}{\mathrm{CPST}}
\begin{document}

\title{On Martin's Pointed Tree Theorem}

\author{Rupert H\" olzl}
\address{Department of Mathematics, Faculty of Science, National University of Singapore, Block S17, 10~Lower Kent Ridge Road, Singapore 119076, Singapore}
\email{r@hoelzl.fr}

\author{Frank Stephan}\
\address{Department of Mathematics, Faculty of Science, National University of Singapore, Block S17, 10~Lower Kent Ridge Road, Singapore 119076, Singapore}
\email{fstephan@comp.nus.edu.sg}

\author{Liang Yu}
\address{Department of Mathematics, Nanjing University, 22 Hankou Road, Nanjing 210093, P.R. China}
\email{yuliang.nju@gmail.com}

\thanks{Rupert H\"olzl was supported by and Frank Stephan was partially supported by the Ministry of Education of Singapore through grant R146-000-184-112 (MOE2013-T2-1-062). Liang Yu was partially supported by the National Natural Science Fund of China through grant~11322112.}

\begin{abstract}
We investigate the reverse mathematics strength of Martin's pointed tree theorem ($\MPT$) and one of its variants, weak Martin's pointed tree theorem ($\wMPT$).
\end{abstract}

\maketitle

\section{Introduction}
A set $A$ of reals is cofinal if for every real $x$ there is some $y\in A$ so that $y\geqT x$. A {\em pointed} tree $T$ is a perfect tree so that for every infinite path $x\in [T]$, $T\leqT x$.

In 1968, Martin \cite{Martin68} proved the following theorem.
\begin{theorem}[Martin \cite{Martin68}]\label{theorem: MPT}
Assume that every game is determined. Then given any set $A\subseteq \omega^{<\omega}$, if $A$ is cofinal, then there is a pointed tree $[T]\subseteq A$.
\end{theorem}

Theorem~\ref{theorem: MPT}  builds a significant connection between descriptive set theory and recursion theory. It has been a central goal in descriptive set theory to prove lower bounds on the consistency strength of some descriptive set theory theorems (see, for example, Harrington \cite{Harrington78} or Koellner and Woodin \cite{KW10}). 
Despite the seemingly simple form of Theorem \ref{theorem: MPT}, the proof of its consistency relative to $\mathrm{ZF}$ requires the 
existence of infinitely many Woodin cardinals; which is far beyond the strength of $\mathrm{ZF}$.
One of the reasons for the importance of Theorem~\ref{theorem: MPT} is that it was used by Slaman and Steel~\cite{SS88} as a critical tool in their study of Martin's conjecture, one of the central open problems in recursion theory.

We study a natural version of Martin's theorem and a variant, $\wMPT$, before a recursion theory background.

\begin{definition}
{\em Martin's pointed tree theorem}, $\MPT$, states that given a tree $T\subseteq \omega^{<\omega}$, if $[T]$ is cofinal, then $T$ has a pointed subtree.
\end{definition}

\begin{definition}
{\em Weak Martin's pointed tree theorem}, $\wMPT$, states that given a tree $T\subseteq \omega^{<\omega}$, if $[T]$ is cofinal, then $T$ has a perfect subtree.
\end{definition}

In this paper, we are mainly interested in the reverse mathematical strengths of these statements.

 Reverse mathematics is used to gauge the complexity of mathematical theorems by determing precisely which axioms are needed to prove a given theorem.
 For example, Martin and Steel \cite{MarSte89} proved the conclusion of Theorem~\ref{theorem: MPT} under the hypothesis that there are infinitely many Woodin cardinals. A typical question in reverse mathematics would be whether this hypothesis is necessary for the proof of a given statement; and in the case of Theorem~\ref{theorem: MPT} it indeed is, as shown by Koellner and Woodin~\cite{KW10}.
 
When studying reverse mathematics before a recursion theory background one often focuses on  second order arithmetical theories.  In other words one focuses on the question of which theorems can be proven assuming only a certain subset of second order arithmetical axioms. In this context, there exists a group of sets of axioms, the so-called \textit{big five}. These sets of axioms are distinguished from others in that most ``usual" mathematical theorems were proven to be equivalent to one of these five. This is why most researchers in the area consider the \textit{big five} systems to be of central importance for the subject, and why they have received much attention.
In this paper, we use the \textit{big five} to measure the strength of $\MPT$ and $\wMPT$.

Based on Martin's proof of Theorem \ref{theorem: MPT}, it seems that for a proof of $\MPT$ we need $\utilde{\mathbf{\Delta}}^0_4$-$\mathrm{DET}$. However, using results from \cite{HK75} and \cite{MS15}, we show that over $\ACA$ we have that $\MPT$ is equivalent to $\ATR$. As a consequence, the question arises whether we can prove the same equivalence over weaker axiom systems. We prove that, over $\RCA$, $\MPT$ does not even imply $\WKL$. So $\MPT$ can be viewed as a natural theorem incomparable with $\WKL$  and $\ACA$ but ``joining" $\ACA$ to $\ATR$. However the question of whether the equivalence can be proven over $\WKL$ remains open.

$\wMPT$ is obviously implied by $\MPT$. It is also related to other classical results in descriptive set theory. The perfect set theorem ($\PST$) says that  every  uncountable set has a perfect subtree. In the reverse mathematics setting $\PST$ corresponds to $\RPST$, the statement that {\em every uncountable closed set has a perfect subset}. A natural analogue of $\wMPT$ in descriptive set theory is the statement that {\em every  cofinal set has a perfect subset }($\CPST$). It turns out (see Solovay \cite{So70} and Chong and Yu~\cite{CY06}) that $\PST$ and $\CPST$ have the same consistency strength.  Obviously $\RPST$ also implies $\wMPT$. Simpson~\cite{Simp09} shows that, over $\ACA$,  $\RPST$ is equivalent to $\ATR$. However, we prove that $\wMPT$ is strictly weaker than $\ATR$ (and therefore than $\RPST$) and incomparable with $\WKL$ and $\ACA$. Hence we have two mathematical statements which have the same consistency strength but different reverse mathematics strength. Another interesting conclusion is that $\wMPT$ can be viewed as a natural example of a theorem that is not equivalent to any of the \textit{big five}.

We also would like to point out the interesting technique used to prove  Theorem \ref{theorem: wmpt not aca}, which states that, over $\WKL$, $\wMPT$ does not imply  $\ACA$. It demonstrates a natural application of algorithmic randomness theory to reverse mathematics. In fact, in this article, the usage of genericity in all proofs could be replaced with randomness; we keep the genericity because it simplifies the proofs. However, the usage of randomness in the proof of Theorem~\ref{theorem: wmpt not aca} seems necessary since no generic real can be hyperimmune-free. Actually, randomness theory usually provides stronger results, though it requires more sophisticated proofs. 
For example, one can use algorithmic randomness theory to prove  Theorem~\ref{theorem: over rca mpt not imply wkl} over $\mathrm{WWKL}_0$ --- a principle that is very close to $\WKL$ and that essentially says that random reals exist.

\bigskip

We organize the paper as follows: In section 2 we review some background knowledge; in section 3 we investigate $\MPT$ and $\wMPT$ over $\RCA$; in section 4 we investigate them over $\WKL$; and in section 5 over $\ACA$.


\section{Preliminaries}
\subsection{General notations}
For every real $x\in \omega^{\omega}$, let $\bf{x}$ be the Turing degree of $x$.

If $\Phi$ is a Turing functional, then we use $\Phi^{x\uh n}[m]$ to denote the finite string computed from oracle $x$ at stage $m$ with use $n$. 

If $x  \leqT y$ via a total Turing functional, then we write $x  \leqtt y$ and say that $x$ is truth-table reducible to $y$.

We refer the reader to Lerman~\cite{Ler83} and Odifreddi~\cite{Odi} for more recursion theoretical background.

Given a tree $T$, we use $[T]$ to denote the collection of the infinite paths through~$T$. For a finite string $\sigma$, let $[\sigma]$ denote the collection of reals extending $\sigma$.


A set $D\subseteq 2^{\omega}$ is {\em dense} if for every $\sigma$ there is some $\tau\in D$ so that $\tau\succ \sigma$, i.e., $\tau$ is an extension of $\sigma$.

A real $g\in 2^{\omega}$ is {\em arithmetically generic}, if for every arithmetical dense set $D\subseteq 2^{<\omega}$ there is some $n$ so that either
\begin{itemize}
\item  $x\uh n\in D$; or
\item $\forall \sigma\,(\sigma\succ x\uh n\rightarrow \sigma\not\in D)$.
\end{itemize}


A real $x$ is {\em hyperimmune-free} if every $x$-recursive function is dominated by a recursive function. It is obvious that if $x$ is hyperimmune-free and $y\leqT x$, then $y\leq_{tt} x$. By Jockusch and Soare's Hyperimmune-Free Basis Theorem~\cite{JockSo72}, every nonempty $\Pi^0_1$ subset of $2^{\omega}$ contains a hyperimmune-free real.

We say that $x\gg y$ if there is a real $z\leqT x$ in $2^{\omega}$ so that $\forall e(z(e)\neq \Phi_e^y(e))$.  There is a nonempty $\Pi^0_1$ subset of $2^{\omega}$ in which every real $x$ has the property $x\gg \emptyset$. So by the Hyperimmune-Free Basis Theorem there is a hyperimmune-free real $x\gg \emptyset$.

A partial function $p:\omega\to \omega$ is recursively bounded if there is a recursive function $f:\omega\to \omega$ so that for every $n$, $p(n)\downarrow\rightarrow p(n)<f(n)$. The following result should be well known.

\begin{proposition}[Folklore]\label{theorem: total recursive}
If $x\gg \emptyset$, then for every partial recursive function $\Phi:\omega\to \omega$ which is recursively bounded, there is a total $x$-recursive function $g$ extending $\Phi$.
\end{proposition}
\begin{proof}
Let a partial recursive function $\Phi:\omega\to \omega$ with a recursive bound~$f$ be given. Define a partial recursive function $\Psi:\omega^2 \to 2$ so that
\[\Psi(e,\langle n,m\rangle )= \left\{
\begin{array}{r@{\quad\quad}l}
1 & \Phi(n)\downarrow =m , \\
0 & \exists k<f(n)(k\neq m\wedge \Phi(n)\downarrow =k),\\
\uparrow & \mbox{otherwise.}\\
\end{array}
\right.\]

 By the $s$-$m$-$n$-Theorem, there is a recursive function $h$ so that 
 \[\Psi(e,\langle n,m\rangle )=\Phi_{h(\langle n,m \rangle)}(e).\] 
 Let $z\leqT x$ be in $2^{\omega}$ so that $\forall e(z(e)\neq \Phi_e(e))$. Define $g(n)=m$ if~$m$ is the least number $<f(n)$  so that $z(h(\langle n,m\rangle))=0$ if any; otherwise let $g(n)=0$.  Obviously $g$ is a total $x$-recursive function extending $\Psi$.
\end{proof}

Note that for every recursive tree $T\subseteq 2^{<\omega}$ and real $x\gg \emptyset$, if $[T]$ is not empty, then there must be some $y\leqT x$ such that $y \in [T]$.

We need the following technical lemma to build a perfect tree by projection. Note that for the purpose of this lemma a tree is a subset of $\omega^{<\omega}\times \omega^{<\omega}$. The motivation is that we will later apply the lemma to game theoretic trees, for games with two players making moves alternately. 
For a tree $T$ of this form, we say that $(\sigma_1,\tau_1)$ extends $(\sigma_0,\tau_0)$ if $\sigma_1$ extends $\sigma_0$ and $\tau_1$ extends $\tau_0$. Then it is also clear what it means for such a $T$ to be perfect.
\begin{lemma}\label{lemma: technique to construct perfect}
Suppose that $T\subseteq \omega^{<\omega}\times \omega^{<\omega}$ is a perfect tree so that for every $(f_0,g_0),(f_1,g_1)\in [T]$, if $g_0=g_1$, then $f_0=f_1$.  Then there is a perfect tree $S_1\leqT T$ so that for every infinite path $g\in [S_1]$, there is some $f$ so that $(f,g)\in [T]$.
\end{lemma}
\begin{proof}
Fix a tree $T$ as in the assumption. We $T$-recursively build a helper tree $S_0$ and the desired tree $S_1$ stage by stage. 

\bigskip

At stage $0$, let $(\emptyset,\emptyset)\in S_0$ and $\emptyset\in S_1$.

We assume that at stage $n$, for every $(\sigma,\tau)\in S_0$, there are $\tau_0|\tau_1$ extending $\tau$ so that there are $\sigma_0, \sigma_1 $ extending  $\sigma$ so that $(\sigma_0,\tau_0),(\sigma_1,\tau_1)\in T$.

We claim that the assumption holds at stage $0$.  Since $T$ is perfect, there must exist two distinct paths $(f_0,g_0), (f_1,g_1)\in [T]$. If we always have $g_0=g_1$ then, by the assumption on $T$, $f_0$ must be equal to $f_1$, which is a contradiction to the fact that $T$ is perfect. So fix $g_0$, $g_1$ and some $n$ such that $g_0(n)\neq g_1(n)$. Write $\tau_0=g_0\uh n+1$, $\tau_1=g_1\uh n+1$, $\sigma_0=f_0\uh n+1$, and $\sigma_1=f_1\uh n+1$. 

At stage $n+1$, for every leaf $(\sigma,\tau)\in S_0$, select  $(\sigma_0,\tau_0),(\sigma_1,\tau_1)\in T$ as defined above and put them into $S_0$. Also put $\tau_0,\tau_1$ into $S_1$. 
With the same argument as above for stage~$0$ the assumption remains true at stage $n+1$.

Now it is clear that  $S_1$ is a perfect tree. Suppose that $g\in [S_1]$, then there must be $(\sigma_0,\tau_0)\prec (\sigma_1,\tau_1)\prec \dots $ constructed at stages $0,1,\dots$, respectively, so that $\tau_0\prec \tau_1\prec\dots \prec g$. Let $f=\bigcup_{i\in \omega}\sigma_i$. Then $(f,g)\in [T]$.
Thus $S_1$ is a required. 
\end{proof}

\subsection{Reverse mathematics}
We refer to Simpson~\cite{Simp09} for the background on reverse mathematics. We recall that $\RCA$ is the most basic axiom system for the second order arithmetical theory. The axioms in the stronger system $\WKL$ state that every infinite binary tree has an infinite path. The even stronger system $\ACA$ includes all arithmetical comprehension axioms, and $\ATR$ ensures that arithmetical transfinite recursion is allowed. Together with  $\utilde{\mathbf{\Pi}}^1_1$-$\mathrm{CA}_0$, these systems form the famous \textit{big five} hierarchy.

Given a theory $T$ and a proposition $\varphi$, to show that $T\not\vdash \varphi$, we will use model-theoretical arguments. A model $\mathcal{M}$ for the second order arithmetical language has the form $(N,M,0,1,+,\times,<)$, which is a two-sorted model. The first sort $N$ contains the natural numbers, and the second sort $M$ the subsets of the natural numbers, respectively, that exist in the model at hand.  In this paper, we will always have $N=\omega$ and $M\subseteq \omega^{\omega}$, that is, we only focus on so-called $\omega$-models. 

\subsection{Game theory} We recall the basic game theoretical notions used in this article, and refer to Moschovakis~\cite{Mos09} for more details.

Given a set $A\subseteq \omega^{\omega}$, we define an infinite game $G_A$ with  perfect information as follows: The game has two players labelled {\bf I} and {\bf II}. The  game is played by letting the players choose natural numbers alternately for $\omega$-many 
steps. Each game generates a real   $x=(n_0,m_0,\dots, n_i,m_i,\dots)\in \omega^{\omega}$ where $n_i$ and $m_i$ are the numbers played by {\bf I} and {\bf II}, respectively, at their $i$-th move.
  If $x\in A$, then {\bf I} wins the game. Otherwise, {\bf II} wins.

A strategy  is a function  $h:\omega^{<\omega}\to \omega$.
For a set $A\subseteq \omega^{\omega}$ and the corresponding game $G_A$, if $h$ is {\bf I}'s strategy and {\bf II} plays $g$, then as usual $h* g$ denotes the outcome generated by $h$ and $g$. If $h$ is {\bf II}'s strategy and {\bf I} plays $f$, then $f*h$ denotes the outcome generated by $h$ and $f$.

{\bf I} has a {\em winning} strategy $h$ for the game $G_A$ if for every  $g\in \omega^{\omega}$, the real $h*g\in A$.  {\bf II} has a {\em winning} strategy $h$ for the game $G_A$ if for every $f\in \omega^{\omega}$, the real  $f*h\not\in A$.

A game $G_A$ is {\em determined} if either {\bf I} or {\bf II} has a winning strategy.
Given a class $\Gamma\subseteq \mathscr{P}(\omega^{\omega})$, $\Gamma$-$\DET$ says that $G_A$ is determined for every $A\in \Gamma$.

The following remarkable  connection between game theory and recursion theorem was established by Martin. Call a set $A$ of reals {\em Turing-invariant} if for every $x\in A$, $y\equivT x$ implies $y\in A$.

\begin{theorem}[Martin \cite{Martin68}]
Assume every set is determined. Then every set $A$ of reals that is Turing-invariant and cofinal contains an upper cone with respect to Turing reducibility.
\end{theorem}

\label{pageref_gamma_TD} {\em $\Gamma$-$\mathrm{TD}$} says that if an $A\in \Gamma$ which is  Turing-invariant is also cofinal, then it contains an upper cone of Turing degrees.

The connection between game theory and reverse mathematics was initiated by Blass and Steel.
\begin{theorem}[Blass \cite{Blass72} and Steel \cite{Steel76}]\label{theorem: atr implies 1 det}
Over $\RCA$,  $\ATR$ implies $\utilde{\mathbf{\Pi}}^0_1$-$\DET$.
\end{theorem}

\subsection{Algorithmic randomness}

In this article, the theory of algorithmic randomness will only serve as a tool. We refer the reader to Nies~\cite{Niesbook09} and Downey and Hirschfeldt~\cite{DH10} for details on the topic.

Given a Turing machine $M$, define its Kolmogorov complexity function as $$C_M(\sigma)=\min\{|\tau|\mid M(\tau)=\sigma\}.$$ The universal Turing machine $U$ induces an optimal Kolmogorov complexity function up to a constant. That is, for every Turing machine $M$, there is a constant $c_M$ so that $\forall \sigma(C_U(\sigma)\leq C_M(\sigma)+c_M)$. Usually, $U$ is fixed and the subscript omitted.

A real $x\in 2^{\omega}$ is {\em random} if for every recursive function $f\colon\omega\to \omega$ with \[{\sum_{n\in \omega}2^{-f(n)}<\infty},\] there is a constant $c$ so that for all $n$, $C(x\uh n)\geq n-f(n)-c$. In particular, if $x$ is random, then there exists a $ c$ such that for all $n$, $C(x\uh n)\geq \frac{n}{2}-c$. There is a nonempty $\Pi^0_1$ set that only contains random reals. So if $x\gg \emptyset$, then $x$ computes a random real.

\section{Over $\RCA$}

\begin{lemma}\label{lemma: recursive pointed tree}
 Suppose that $T\subseteq \omega^{<\omega}$ is a recursive tree so that there is a nonrecursive infinite path $x\in [T]$ that is Turing-below some arithmetically generic real $g \in 2^{\omega}$. Then $T$ has a  recursive perfect subtree.
\end{lemma}
\begin{proof}
 Suppose that $T$ is a recursive tree with a nonrecursive infinite path $x\in [T]$ so that $x=\Phi^g$ for some arithmetical real $g$.

Let $D_0=\{\sigma \mid \Phi^{\sigma}\not\in T\}$. Clearly, $D_0$ is arithmetical.   Since $g$ is arithmetically generic and  $\Phi^g\in [T]$, there must be some $n_0$ so that for every $\sigma\succ g\uh n_0$, $\sigma\not\in D_0$ and so $\Phi^{\sigma}\in T$.

Let 
\[D_1=\left\{\sigma\scmid \begin{array}{c}
\sigma\succ g\uh n_0\,\wedge\, \exists n\,\forall m\geq n\, \forall \tau_0\succ \sigma\,\forall \tau_1\succ \sigma\colon\\ (\Phi^{\tau_0}(m)\!\downarrow \wedge\, \Phi^{\tau_1}(m)\!\downarrow\,\rightarrow\, \Phi^{\tau_0}(m)=\Phi^{\tau_1}(m))
\end{array}
\right\}.\]
Since  $g$ is arithmetically generic and  $\Phi^g$ is  not recursive, there must be some $n_1\geq n_0$ so that  for every $\sigma\succ g\uh n_1$, $\sigma\not\in D_1$. So for every $\sigma\succ g\uh n_1$ and $n$, there are $\tau_0\succ \sigma$, $\tau_1\succ \sigma$ and $m>n$ so that $\Phi^{\tau_0}(m)\neq\Phi^{\tau_1}(m)$.

Now, by using this property of $D_1$,  it is routine to construct a recursive perfect tree $S\subseteq [g\uh n_1]$ so that the the set $T_1=\{\Phi^{\sigma}\mid \sigma\in S\}$ is also a recursive perfect tree. By the property of $n_0$, $T_1\subseteq T$.
 \end{proof}

To construct a model satisfying $\MPT$, we have to relativize Lemma \ref{lemma: recursive pointed tree} accordingly.

Given a tree $T\subseteq \omega^{<\omega}$, a real $z$ and an index $i$, let $\HK_i(z,T)$ be a $z\oplus T$-recursive tree so that 
\[[\HK_i(z,T)]=\left\{f\oplus g \scmid 
\begin{array}{c}
g \in [T] \,\wedge\, \forall n\, (f(n)\mbox{ is the least }m\\\mbox{ with }\Phi_i^{g\upharpoonright m}(n)[m] \downarrow   \wedge \Phi_i^{g}(n)=z(n))
\end{array}
\right\}.\]
The idea of $\HK_i(z,T)$ originates from Harrington and Kechris~\cite{HK75}.

 Obviously $\HK_i(z, T)\leqT T \oplus z$. Note that if $\Phi_i^g$ is undefined or different from $z$ then no path of the form $f\oplus g$ can be contained in $\HK_i(z,T)$.  As a consequence, if $z\geqT T$, then it holds for every $f\oplus g\in [\HK_i(z,T)]$ that $g\in [T]$ and $g\geqT z \geqT z\oplus T\geqT \HK_i(z, T)$.

\begin{theorem}\label{theorem: over rca mpt not imply wkl}
Over $\RCA$, $\MPT$ does not imply $\WKL$.
\end{theorem}
\begin{proof}
Choose a sequence $g_0, g_1, g_2,\dots$ of elements of $ 2^{\omega}$ such that $g_0$ is arithmetically generic and such that for all $n$, $g_{n+1}$ is arithmetically generic relative to $g_0\oplus g_1\oplus g_2\cdots \oplus g_n$.
Let  $\mathcal{M}=(\omega, M,\dots)$, where $M=\{x\mid \exists n (x\leqT \oplus _{i\leq n}g_i)\}$.  
Obviously $\mathcal{M}\not\models \WKL$, as $\WKL$ would guarantee the existence of PA-complete sets, but no such sets can exist in  $\mathcal{M}$.

Now assume we are given a tree $T$ that is cofinal in $M$.   Then there must be some $x\in [T]$ and some $n$ so that $T \leqT g_0\oplus g_1\oplus g_2\cdots \oplus g_n\leT x$. Let $j$ be an index of the second reduction, that is, $\Phi_j^x=g_0\oplus g_1\oplus g_2\cdots \oplus g_n$. On the other hand there must exist some $m>0$ so that $x\leqT \oplus _{i\leq m+n}g_i$. Note that $\oplus _{n<i\leq m+n}g_i$ is $\oplus_{i\leq n}g_i$-arithmetically generic. Then $\HK_j(\oplus _{i\leq n}g_i, T)$ is an  $\oplus _{i\leq n}g_i$-recursive tree such that

\begin{enumerate}
\item  for every $f\oplus g\in [\HK_j(\oplus _{i\leq n}g_i,T)]$, $g\in [T]$ and $g\geqT \oplus _{i\leq n}g_i\oplus T\geqT \HK_j(\oplus _{i\leq n}g_i, T)$; and 
\item  there is an $f\oplus g\in [\HK_j(\oplus _{i\leq n}g_i,T)]$, for example $f\oplus x$ for some $f$, so that $\oplus _{i\leq m+n}g_i\geqT  f\oplus g\geT \oplus _{i\leq n}g_i \geqT \HK_j(\oplus _{i\leq n}g_i, T)$.
\end{enumerate}

Note that $\oplus _{n<i\leq m+n}g_i$ is $[\HK_j(\oplus _{i\leq n}g_i,T)]$-arithmetically generic. Relativizing  Lemma \ref{lemma: recursive pointed tree} to $\oplus _{i\leq n}g_i$, we apply it to $ [\HK_j(\oplus _{i\leq n}g_i,T)]$ to obtain a  $\oplus _{i\leq n}g_i$-recursive perfect tree  $S\subseteq \HK_j(\oplus _{i\leq n}g_i,T)$. By the property (1) of $\HK(\oplus _{i\leq n}g_i, T)$, $S$ is pointed.

Note that it follows directly from the definition of $\HK_{(.)}(.,.)$ that if $f_0\oplus g \in [\HK_j(\oplus _{i\leq n}g_i, T)]$ and $f_1\oplus g \in [\HK_j(\oplus _{i\leq n}g_i, T)]$ then $f_0=f_1$. In particular this is true for 
$f_0\oplus g$ and $f_1\oplus g \in S \subseteq \HK_j(\oplus _{i\leq n}g_i,T)$. Since $S$ is a perfect tree, by Lemma \ref{lemma: technique to construct perfect}, there is an $S$-recursive perfect tree $T_1$ so that for every $g\in [T_1]$, there is some $f$ so that $f\oplus g\in [S]$. Note that $T_1\subseteq T$. Moreover, by property (1),   for every $g\in [T_1]$, $g\geqT S\geqT T_1$. So $T_1$ is a pointed subtree of $T$.
\end{proof}

We prove that $\ACA$ does not imply $\wMPT$.

\begin{lemma}\label{lemma: union singletons}
There is a recursive tree $T$ so that  $[T]$ is countable and cofinal for the arithmetical reals, that is, for every arithmetical real $x$ there exists some $y \in [T]$ with~$y \geqT x$.\end{lemma}
\begin{proof}
It is well known (see, for example, Sacks~\cite{Sacks90}) that there is a sequence of uniformly recursive trees $\{T_n\}_{n\in \omega}$ so that for every $n$, $[T_n]$ contains a unique real~$x_n$ of Turing degree $\mathbf{0}^{(n)}$. Let $T=\bigcup_{n\in \omega}n^{\smallfrown}T_n$.
\end{proof}

\begin{proposition}\label{theorem: aca does not imply wmpt}
$\ACA$ does not imply $\wMPT$.
\end{proposition}
\begin{proof}

Let  $\mathcal{M}=(\omega, M, \dots)$, where $M=\{x\mid \exists n\, (x\leqT \emptyset^{(n)})\}$.  Obviously, ${\mathcal{M}\models \ACA}$. 
  Let $T$ be as in Lemma \ref{lemma: union singletons}. Then $T$ is cofinal in $M$. However, $T$ has no perfect subtree. Hence $\mathcal{M}\not\models \wMPT$.
\end{proof}

By Corollary \ref{corollary: wept does not imply mpg}, over $\RCA$, $\wMPT$ does not imply $\MPT$.

\section{Over $\WKL$}

\begin{lemma}\label{lemma: hif pa computes}
Suppose that $T\subseteq 2^{<\omega}$ is a recursive tree and that $x\in [T]$ is a real so that there is some random real $y\leqtt x$. Then for every real $z\gg \emptyset$,  there must be some perfect tree $S\subseteq T$ so that $S\leqT z$. 
\end{lemma}
\begin{proof}
Suppose that $\Phi_e$ is a $tt$-reduction so that $\Phi_e^x=y$. Let $c_0$ be such that for all~$n$ we have $C(y\uh n)\geq \frac{n}{2}+c_0$ and let $f$ be a recursive, increasing function such that for all $n$ it holds that $\Phi_e^{x\upharpoonright f(n)}\uh n [f(n)]=y\uh n$. Without loss of generality we assume that $f(n)>2^n$ for all $n$.

Let $T_1\subseteq T$ be a recursive tree so that $$[T_1]=\{z\in 2^{\omega}\mid z\in [T]\wedge \forall n\,(C(\Phi_e^{z\upharpoonright f(n)}\uh n [f(n)])\geq \frac{n}{2}-c_0)\}.$$  Note that $x\in [T_1]$.

Let $g$ be a recursive function so that $g(0)=0$ and for every $n$, $g(n+1)=f(f(g(n)))$. 

\smallskip

\noindent {\em Claim.} There is an  $n_0$ so that  for all $n> n_0$ and every $\sigma \in T_1\cap 2^{g(n)}$ with $[\sigma]\cap [T_1]\neq\emptyset$,  there are two different $\sigma_0,\sigma_1\in 2^{g(n+1)}\cap T_1$ extending $\sigma$  so that $[\sigma_0]\cap [T_1]\neq \emptyset$ and  $[\sigma_1]\cap [T_1]\neq \emptyset$. 

\smallskip

\noindent {\em Subproof.} If not, then for every $m$ there is an $n\geq m$ and some $\sigma_0\in 2^{g(n)}\cap T_1$ with $[\sigma_0]\cap [T_1]\neq\emptyset$,   and a unique  string $\sigma$ in $g(n+1)$ extending $\sigma_0$ such that $[\sigma]\cap [T_1]\neq \emptyset$.  Then we build a Turing machine $M$ as follows: A pair $(\nu_1,\tau)$ is enumerated into $M$ 
if and only if 
\begin{enumerate}
\item $\nu_1\in 2^{g(n)}\cap T_1$ and $|\tau|=|f(g(n))|$ for some $n$; and
\item  there is a unique  string $\nu_2 \in 2^{g(n+1)}\cap T_1$ extending $\nu_1$ such that  for every $\nu_3\in 2^{g(n+1)}\cap T_1$ extending $\nu_1$ with $\nu_3 \not= \nu_2$ we have $[\nu_3]\cap [T_1]=\emptyset$; and
\item  for the $\nu_2$ above we have $\tau=\Phi_e^{\nu_2\upharpoonright g(n+1)}[g(n+1)]$.
\end{enumerate}

Recall that $f(n)>2^n$ for every $n$. So if $M(\sigma)=\tau$, then $C_M(\tau)\leq \log |\tau| $. Hence $C(\tau)\leq \log |\tau|+c_M$ for some constant $c_M$.

Then for every $m$, there is an $n>m$ and some $z\in [T_1]$ such that \[C(\Phi_e^{z\uh g(n+1)}\uh f(g(n)))\leq \log f(g(n))+c_M.\] This contradicts the choice of $T_1$. \hfill $\Diamond$

\smallskip

Now define a partial recursive function $\varphi\colon 2^{<\omega}\times \omega \rightarrow 2^{<\omega}$ as follows: For every $n$,  $\sigma\in 2^{g(n)}$ and $m\leq 2^{g(n+1)}$ let $\varphi(\sigma,m)$ be  the $m$-th finite string $\tau$  in $2^{g(n+1)}$ for which we detect that $[\tau]\cap [T_1]=\emptyset$. By Proposition \ref{theorem: total recursive}, there is a total $z$-recursive function $f$ extending $\varphi$.

For every $\sigma \in 2^{g(n)}$, let $$S_{\sigma}=\{\tau\succ \sigma\mid \tau \in 2^{g(n+1)}\}\setminus \{f(\sigma,m)\mid m\leq 2^{g(n+1)-g(n)}-2\}.$$

Then the sequence $\{S_{\sigma}\}_{\sigma\in 2^{<\omega}}$ is $z$-recursive.

By the Claim,   for every $n> n_0$ and $\sigma \in T_1\cap 2^{g(n)}$ with $[\sigma]\cap [T_1]\neq\emptyset$,
\begin{enumerate}
\item $|S_{\sigma}|\geq 2$; and
\item if  $\tau\in S_{\sigma}$, then $[\tau]\cap [T_1]\neq\emptyset$; and
\item any two different strings in $S_{\sigma}$ are incompatible. 
\end{enumerate}

Now, using these facts, it is easy to $z$-recursively construct a perfect tree
\[S\subseteq T_1\subseteq T.\qedhere\]
\end{proof}

\begin{theorem}\label{theorem: wmpt not aca}
Over $\WKL$, $\wMPT$ does not imply $\ACA$.
\end{theorem}
\begin{proof}
Let $x_0=\emptyset \ll x_1 \ll x_2\ll\dots $ be a sequence of hyperimmune-free reals. We can see inductively that such a sequence exists, because we can at step $n$ build an $x_{n-1}$-recursive tree $P_n$ such that all of its paths $y$ have $y \gg x_{n-1}$; we can then apply the Hyperimmune-Free Basis Theorem relative to $x_{n-1}$ to get $x_n$.

Let $\mathcal{M}=(\omega,M,\dots)$ where $M=\{y\mid \exists n\,(y\leqT x_n)\}$.

Obviously, $\mathcal{M}\models \WKL$ but  $\mathcal{M}\not\models \ACA$. 

Let $T\subseteq \omega^{<\omega}$ be a  tree that is cofinal in $M$. Let $T_1\subseteq 2^{<\omega}$ so that $\sigma \in T_1$ if and only if there is a $\tau\in T$ so that $\sigma$ is of the form $0^{\tau(0)}10^{\tau(1)}\dots 0^{\tau(|\tau|)}1 0\dots 0$. Then for every $\{\mathbf{x}\mid x\in [T_1]\setminus [T]\}=\mathbf{0}$ and $\{\mathbf{x}\mid x\in [T]\}\setminus \{\mathbf{x}\mid x\in [T_1]\}=\emptyset$.  So $T_1$ is  also  cofinal in $M$. $\WKL$ implies that a $T_1$-random real $y$ exists; since $T_1$ is cofinal, there must be some $x\in [T_1]$ with $x \geqT y$ in $M$. Since $x$ is hyperimmune-free, in fact $y\leqtt x$.  Fix a number $n$ so that $x_n\gg  T_1$.  Applying Lemma \ref{lemma: hif pa computes} by relativizing it to $T_1$, there must be some $x_n$-recursive perfect tree $S_1\subseteq T_1$. Then it is easy to see that there must be some $x_n$-recursive perfect tree $S\subseteq T$. Thus $S\in M$. 
\end{proof}

Note that by Corollary \ref{corollary: wept does not imply mpg}, over $\WKL$, $\wMPT$ does not imply $\MPT$.
\begin{question}
Over $\WKL$, does $\MPT$ imply $\ACA$? 
\end{question}

\section{Over $\ACA$}

The following lemma is easy.
\begin{lemma}\label{lemma: pointed implies upper cone}
Over $\RCA$, if $T$ is a pointed tree, then for every real $x\geqT T$, there is a $y\in [T]$ so that $x\equivT y$.
\end{lemma}
\begin{proof}
As $T$ is in particular perfect, we can choose a path $y$ such that $y$ encodes $x$ by branching in $T$ according to the bits of $x$. Then
\(x \geqT x\oplus T \geqT y \geqT y \oplus T \geqT x.\qedhere\)
\end{proof}

The following important theorem can be used to transfer results about $\Sigma^0_3$ sets to recursive trees.
\begin{theorem}[Harrington and Kechris \cite{HK75}]\label{theorem: hk translation theorem}
Over $\RCA$, for every $\Sigma^0_3$ class $A$, there is a recursive tree $T$ so that $\{\mathbf{x}\mid x\in [T]\}=\{\mathbf{x}\mid x\in A\}$. Moreover, the proof can be relativized.
\end{theorem}

Recall the definition of  {\em $\Gamma$-$\mathrm{TD}$} from page~\pageref{pageref_gamma_TD}.
 Montalb{\' a}n and Shore proved the following theorem by combining Theorem \ref{theorem: hk translation theorem} with a number of classical recursion theory results.
\begin{theorem}[Montalb{\' a}n and Shore \cite{MS15}]\label{theorem? MS 3 TD}
Over $\ACA$, $\utilde{\mathbf{\Sigma}}^0_3$-$\mathrm{TD}$ implies $\ATR$.
\end{theorem}

Now we are ready to show that  $\MPT$ implies $\ATR$.
\begin{proposition}\label{proposition: mpt implies atr}
Over $\ACA$, $\MPT$ implies $\ATR$.
\end{proposition}
\begin{proof}
By Theorem \ref{theorem? MS 3 TD}, it is sufficient to prove that $\MPT$ implies  $\utilde{\mathbf{\Sigma}}^0_3$-$\mathrm{TD}$. 

We only prove the lightface version. Given any  $\Sigma^0_3$, cofinal, and Turing-invariant set $A$, by Theorem \ref{theorem: hk translation theorem}, there is a recursive tree $T$ so that $\{\mathbf{x}\mid x\in [T]\}=\{\mathbf{x}\mid x\in A\}$. So $[T]$ is also cofinal. By $\MPT$, $[T]$ has a pointed subtree. By Lemma \ref{lemma: pointed implies upper cone},  for every real $x\geqT T$, there is a $y\in [T]$ so that $x\equivT y$.  So $A$  contains  an upper cone of Turing degrees.
\end{proof}
\begin{remark}
Note that, over $\RCA$, $\utilde{\mathbf{\Sigma}}^0_3$-$\mathrm{TD}$ does not imply $\wMPT$. This is because $\RCA$ has a model $\mathcal{M}$ that consists only of recursive reals. Then $\mathcal{M}$ satisfies $\utilde{\mathbf{\Sigma}}^0_3$-$\mathrm{TD}$ vacuously, but it contains a recursive tree $T$ consisting exactly of  one recursive path. Such a $T$ is then vacuously cofinal, but not perfect. So $\mathcal{M}$ does not satisfy~$\wMPT$.
\end{remark}

In the following theorem we will show that $\ATR$ implies $\MPT$. For this purpose, let $T\subseteq \omega^{<\omega}$ be any tree and define the $T$-recursive tree $\HK(T)$ as follows.
\[
[\HK(T)]=\left\{(i^{\smallfrown}f\oplus g) \oplus h\scmid  
\begin{array}{c}
g \in [T]  \,\wedge\,  h=(i^{\smallfrown}f\oplus g) * \Phi_i^g\,\wedge\,\\  
 \forall n\, (f(n)\mbox{ is the least }m\mbox{ with }\Phi_i^{g\upharpoonright m}(n)[m] \downarrow)
\end{array}
\right\}
\]
The idea for $\HK(T)$ is again taken from Harrington and Kechris~\cite{HK75}.

\begin{theorem}\label{theorem: atr implies mpt}
$\ATR$ implies $\MPT$.
\end{theorem}
\begin{proof}
Given any cofinal tree $T$, $\HK(T)$ is clearly also cofinal. By Theorem \ref{theorem: atr implies 1 det}, $\HK(T)$ is determined.  It is well known  (see Martin~\cite{Martin68} or Montalb\'an and Shore~\cite{MS15}) that if $\HK(T)$ is cofinal,  then {\bf II} cannot  have a winning strategy. So {\bf I} has a winning strategy, say $w$. Let $x\in 2^{\omega}$ be a real with $x\equivT w$. 

Let $S$ be a $w$-recursive tree so that $[S]=\{w*(z\oplus x)\mid z\in 2^{\omega}\}$. Then
\begin{enumerate}
\item if $(i^{\smallfrown}f_0\oplus g) \oplus (z_0\oplus x)\in [S]$ and $(i^{\smallfrown}f_1\oplus g) \oplus (z_1\oplus x)\in [S]$, then $f_0=f_1$ and $z_0=z_1$. Together with Lemma \ref{lemma: technique to construct perfect} this implies that there is a $w$-recursive perfect tree $T_1\subseteq T$.

\item for every $(i^{\smallfrown}f\oplus g) \oplus (z\oplus x)\in [S]$, we have that $g\in [T]$ and $g\geqT z\oplus x\geqT w \geqT S$. This implies that $T_1$ is pointed.
\end{enumerate} 
Thus $\ATR$ implies $\MPT$.
\end{proof}

However, even over $\ACA$, $\wMPT$ is strictly weaker than $\ATR$.

\begin{theorem}\label{theorem: wmpt does not imply art}
Over $\ACA$, $\wMPT$ does not imply $\ATR$.
\end{theorem}
\begin{proof}

Let $g_0 \in 2^{\omega}$ be arithmetically generic and for every $n$ let $g_{n+1} \in 2^{\omega}$ be  arithmetically generic relative to $g_0\oplus g_1\oplus g_2 \oplus \ldots \oplus g_n$. Let  $\mathcal{M}=(\omega, M,\dots)$, where $M=\{x\mid \exists n (x\leqT (\oplus _{i\leq n}g_i)^{(n)})\}$.  Obviously $\mathcal{M}\models \ACA$, but  $\mathcal{M}\not\models \ATR$ as ${\emptyset}^{(\omega)}$ is not in $M$.

Now let every tree $T\subseteq \omega^{<\omega}$, which is cofinal in $M$, be given.   Then there must be some real $x\in [T]$ and some $n$ so that $T \leqT \emptyset^{(n)}\oplus g_0\oplus g_1\oplus g_2\cdots \oplus g_n\leT x$ and $x$ is not arithmetical in  $ g_0\oplus g_1\oplus g_2\cdots \oplus g_n$.
Fix~$j$ such that $\Phi_j^x=\emptyset^{(n)} \oplus g_0\oplus g_1\oplus g_2\cdots \oplus g_n$. On the other hand, there is some $m>0$ so that $x\leqT \emptyset^{(m+n)}\oplus (\oplus _{i\leq m+n}g_i)$. Note that $\oplus _{n<i\leq m+n}g_i$ is $\emptyset^{(m+n)}\oplus (\oplus_{i\leq n}g_i)$-arithmetically generic. Then $\HK(\oplus _{i\leq n}g_i, T)$ is a  $\emptyset^{(n)}\oplus(\oplus _{i\leq n}g_i)$-recursive tree so that

\begin{enumerate}
\item  for every $f\oplus g\in [\HK_j(\oplus _{i\leq n}g_i,T)]$, $g\in [T]$ and $g\geqT \emptyset^{(n)}\oplus(\oplus _{i\leq n}g_i)\oplus T\geqT \HK_j(\oplus _{i\leq n}g_i, T)$; and 
\item  there is an $f\oplus g\in [\HK_j(\oplus _{i\leq n}g_i,T)]$, for example $f\oplus x$ for some $f$, so that $\emptyset^{(m+n)}\oplus(\oplus _{i\leq m+n}g_i)\geqT  f\oplus g\geT \emptyset^{(n)}\oplus(\oplus _{i\leq n}g_i) \geqT \HK_j(\oplus _{i\leq n}g_i, T)$.
\end{enumerate}

Note that $\oplus _{n<i\leq m+n}g_i$ is $[\HK_j(\emptyset^{(n)}\oplus(\oplus _{i\leq n}g_i),T)]$-arithmetically generic. Relativizing  Lemma \ref{lemma: recursive pointed tree} to $\emptyset^{(n)}\oplus(\oplus _{i\leq n}g_i)$ together with the fact that  $x$ is not arithmetical in  $\oplus _{i\leq n}g_i$, we apply it to $[\HK_j(\emptyset^{(n)}\oplus(\oplus _{i\leq n}g_i),T)]$ and obtain that there exists a  $\emptyset^{(m+n)}\oplus(\oplus _{i\leq n}g_i)$-recursive perfect tree  $S\subseteq \HK_j(\oplus _{i\leq n}g_i,T)$. 

Note that  $f_0\oplus g \in S\subseteq  [\HK(\oplus _{i\leq n}g_i, T)]$ and $f_1\oplus g \in S\subseteq [\HK(\oplus _{i\leq n}g_i, T)]$ implies $f_0=f_1$. Using the fact that $S$ is a perfect tree, by Lemma \ref{lemma: technique to construct perfect}, there is a $\emptyset^{(m+n)}\oplus(\oplus _{i\leq n}g_i)$-recursive perfect tree $T_1$ so that for every $g\in [T_1]$, there is some $f$ so that $f\oplus g\in [S]$. Note that $T_1\subseteq T$.  So $T$ has a perfect subtree in $M$.

Thus $\mathcal{M}\models \wMPT$.
\end{proof}

\begin{corollary}\label{corollary: wept does not imply mpg}
Over $\ACA$, $\wMPT$ does not imply $\MPT$.
\end{corollary}

We remark that by a method similar to the proof of Theorem \ref{theorem: wmpt does not imply art} it can be shown that, even over $ \utilde{\mathbf{\Delta}}^1_1$-$\mathrm{CA}_0$, $\wMPT$ does not imply $\ATR$.

\smallskip

Figure~\ref{fig:diagram} gives an overview of the obtained results.

\begin{figure}[tb]
\begin{tikzpicture}[scale=.6,auto=left,every node/.style={fill=black!15},y=-2cm]
  \node (ATR) at (0,0) {$\ATR \equiv \ACA + \MPT$};

  \node (ACA) at (-6,2) {$\ACA$};
  \node (WKL) at (-6,6) {$\WKL$};

  \node (WKLwMPT) at (0,4) {$\WKL + \wMPT$};


  \node (MPT) at (6,2) {$\MPT$};
  \node (wMPT) at (6,6) {$\wMPT$};

  \node (RCA) at (0,8) {$\RCA$};

  \foreach \from/\to in {ATR/ACA,ACA/WKL,WKL/RCA,ATR/MPT,MPT/wMPT,wMPT/RCA,ATR/WKLwMPT,WKLwMPT/WKL,WKLwMPT/wMPT}
  \draw [->,thick] (\from) -- (\to);
\end{tikzpicture}
\caption{All arrows represent strict implications.}
\label{fig:diagram}
\end{figure}
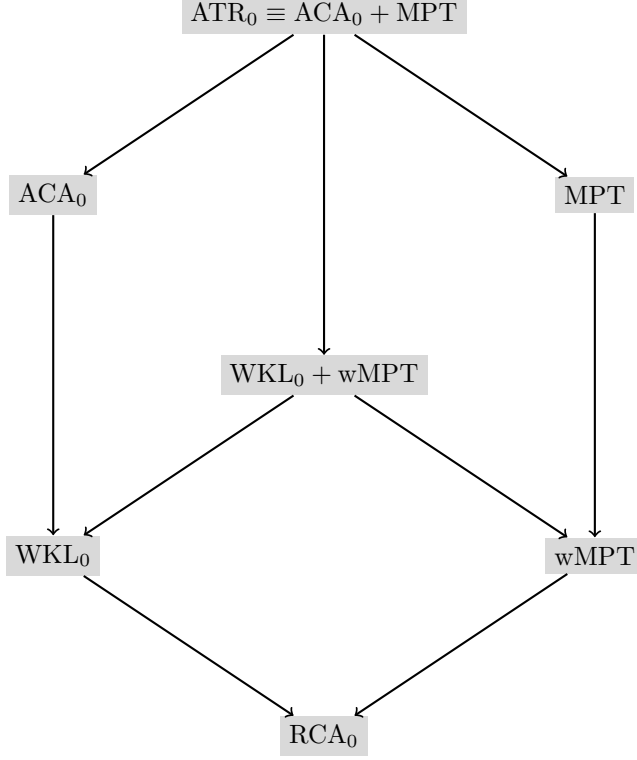

\bibliographystyle{plain}
\bibliography{references}

\end{document}